\documentclass[a4paper,12pt]{article}
\usepackage{hyperref}
\usepackage{amsmath}
\usepackage{amssymb}
\usepackage{amsthm}
\usepackage{mathrsfs}
\newtheorem{theorem}{Theorem}[section]
\newtheorem{lemma}[theorem]{Lemma}

\newtheorem{definition}[theorem]{Definition}
\newtheorem{corollary}[theorem]{Corollary}
\newtheorem{remark}[theorem]{Remark}
\newtheorem{example}[theorem]{Example}

\begin{document}

{\bf On the consistency of cell division processes} 

\vspace{0.5cm}

{Werner Nagel} 
\footnote{Friedrich-Schiller-Universit\"at Jena, Fakult\"at f\"ur
Mathematik und Informatik, D-07737 Jena, Germany. werner.nagel@uni-jena.de} 
and  {Eike Biehler}

{\bf Abstract}\\
For a class of cell division processes, generating tessellations of the Euclidean space $\mathbb{R}^d$, spatial consistency is investigated. This addresses the problem whether the distribution of these tessellations, restricted to a bounded set $V$, depends on the choice of a larger region $W\supset V$ where the construction of the cell division process is performed. This can also be understood as the problem of boundary effects in the cell division procedure. In \cite{Nagel:2005} it was shown that the STIT tessellations are spatially consistent
There were hints that the STIT tessellation process might be the only translation-invariant cell division process that has such a consistency property. In the present paper it is shown that, within a reasonable wide class of cell division processes, the STIT tessellations are the only ones that are consistent.
\\


AMS: {60D05}

\section{Introduction}
In Stochastic Geometry the well established models for random tessellations of the Euclidean space $\mathbb{R}^d$, $d \geq 2,$ are the Poisson-Voronoi tesselations and the Poisson hyperplane tessellations, see \cite{SKM, SW}. Moreover, there are many suggestions in the literature to construct tessellations by sequential division of the cells, i.e. of the polytopes which constitute a tessellation. A systematization including many of such constructions was recently given in \cite{Cowan}. Usually, these tessellations are constructed in a bounded window $W\subset \mathbb{R}^d$. This yields a key problem for this kind of constructions: Are the tessellations consistent in space, i.e. does the distribution of the resulting tessellation depend on the window where the construction is performed?
More precisely, if $Y(W,t)$ and $Y(V,t)$ are the random tessellations generated by cell division until time $t>0$ in bounded windows $V\subset W$, are then $Y(W,t)\cap V$ and $Y(V,t)$ identically distributed? The consistency of a model implies the existence of a tessellation $Y(t)$ of the whole space $\mathbb{R}^d$ such that the restrictions $Y(t)\cap W$ have the same law as  $Y(W,t)$.
In the present paper we consider a certain class of cell division processes, and the main result is that in this class the STIT tessellations (introduced in \cite{Nagel:2005}) are the only ones which are consistent.

\section{Random tessellations and consistency}

For the $d$-dimensional Euclidean space $\mathbb{R}^d$, $d \geq 2,$ denote by $\partial B$, $int (B)$ and  $cl(B)$ the boundary, the interior and the topological closure, respectively, of a set $B\subset {\mathbb R}^d$.
Let  $[\mathcal{H},\mathfrak{H}]$ denote the measurable space of  hyperplanes in $\mathbb{R}^d$ with its Borel $\sigma$-algebra $\mathfrak{H}$ w.r.t. the topology of of closed convergence for closed subsets of $\mathbb{R}^d$, see \cite{SW}. For a set $B\subset \mathbb{R}^d$ we write $[B]=\{ h\in {\cal H}:\, B\cap h\not= \emptyset \}$. 
Further, let $\mathfrak{P}$ denote the set of all polytopes (i.e. convex hulls of finite point sets)  with interior points in $\mathbb{R}^d$.
A set $\{ C_1,C_2,\ldots \}$ with $C_i\in \mathfrak{P}$ is a tessellation if $\bigcup_{i=1}^\infty C_i = {\mathbb R}^d$ and  $int (C_i)\cap int (C_j)=\emptyset$ for $i\not= j$. Moreover, a locally finiteness condition must be satisfied,  $\# \{ i:C_i \cap B \not= \emptyset \} <\infty$ for all bounded $B\subset {\mathbb R}^d$, i.e. the number of polytopes intersecting a bounded set is finite.

A tessellation can be considered as a set $\{ C_1,C_2,\ldots \}$ of polytopes -- referred to as the cells -- as well as a closed set  $\bigcup_{i=1}^\infty \partial C_i \subset {\mathbb R}^d$, the union of the cell boundaries. There is an obvious one-to-one relation between both descriptions of a tessellation, and also the $\sigma$-algebras which are used for them can be related appropriately.
Let ${\mathbb T}$ denote the set of all tessellations of ${\mathbb R}^d$. By $y\in {\mathbb T}$ we mean the closed set of cell boundaries of the tessellation $y$. Then ${\mathbb T}$ can be endowed with the Borel $\sigma$-algebra ${\cal B}({\mathbb T})$ of the topology of closed convergence.
A random tessellation $Y$ is a random variable with values in $({\mathbb T}, {\cal B}({\mathbb T}))$.  For $W\in \mathfrak{P}$, the set of tessellations restricted to $W$ is denoted ${\mathbb T}\wedge W$. In particular, for $y\in {\mathbb T}$ we have $y\cap W\in {\mathbb T}\wedge W$, and the boundary of $W$ does not belong to this restricted tessellation. Here $W$ is referred to as a window.  

By  $\stackrel{D}{=}$ we denote the identity of the distributions of random variables. Our investigation of consistency of random tessellations will be based on the following proposition. In \cite{SW}, Theorem 2.3.1, a more general form is given; we specify it here for ${\mathbb R}^d$.

\begin{theorem}\label{conssw} (Schneider and  Weil)  Let $(Z_i: i \in \mathbb{N})$ be a sequence of random closed sets in ${\mathbb R}^d$, and  $(G_i: i \in \mathbb{N})$ a sequence of open, bounded sets  with $cl(G_i) \subset G_{i+1}$ for $i \in \mathbb{N}$ and $\bigcup_{i=1}^\infty G_i = {\mathbb R}^d$. If  $Z_m \cap cl(G_i) \stackrel{D}{=} Z_i$ for all $m > i$, then there exists a random closed set $Z$ in ${\mathbb R}^d$ with $$Z \cap cl(G_i) \stackrel{D}{=} Z_i$$ for all $i \in \mathbb{N}$. 
\end{theorem}

This assertion leads us to the following definition. 

\begin{definition}\label{defcons}
A family $(Y(W): W\in \mathfrak{P})$ of random tessellations with  $Y(W)\in {\mathbb T}\wedge W$ is called consistent if and only if for any two windows $V, W \in \mathfrak{P}$ with 
 $V \subset W$ holds
$$Y( V) \stackrel{D}{=} Y( W) \cap V .$$
\end{definition}

Obviously, if $Y$ is a random tessellation of ${\mathbb R}^d$, then $(Y(W):= Y\cap W,\ W\in \mathfrak{P})$ is a consistent family of tessellations. On the other hand, Theorem \ref{conssw} yields that for any consistent family $(Y(W),\ W\in \mathfrak{P})$ exists a random tessellation $Y$ such that $Y\cap W \stackrel{D}{=} Y( W) .$

We will go a step further and study continuous-time processes of tessellations. Therefore, also the consistency of finite-dimensional distributions of these processes have to be considered. 

\begin{definition}\label{defconsproc}
For any $W \in \mathfrak{P}$ let $(Y(t,W), t > 0)$ be a random process of tessellations with values in ${\mathbb T}\wedge W$.
The family of tessellation processes $((Y(t,W), t > 0): W \in \mathfrak{P})$ is called consistent in space if and only if for any two windows $V, W \in \mathfrak{P}$ with 
 $V \subset W$ and for all $0 < t_1 < ... < t_n,\ n \in \mathbb{N}$ holds
  $$(Y(t_1,V), ..., Y(t_n, V)) \stackrel{D}{=} (Y(t_1,W) \cap V, ..., Y(t_n, W) \cap V)  .$$
\end{definition}

Again, if $(Y(t), t > 0)$ is a random process of tessellations of ${\mathbb R}^d$, then $((Y(t,W), t > 0): W \in \mathfrak{P})$ is a consistent family of tessellation processes.
Vice versa, if $((Y(t,W), t > 0): W \in \mathfrak{P})$ is a consistent family of tessellation processes, then for all $0 < t_1 < ... < t_n,\ n \in \mathbb{N}$ exist tessellations $Y(t_1),\ldots , Y(t_n)$ with 
$$(Y(t_1,W), ..., Y(t_n, W)) \stackrel{D}{=} (Y(t_1) \cap W, ..., Y(t_n) \cap W)  $$
for all $ W \in \mathfrak{P}$.
Because  the laws of $(Y(t_1),\ldots , Y(t_n)) $ with $0 < t_1 < ... < t_n,\ n \in \mathbb{N},$ form a projective family of distributions, and the measurable space $({\mathbb T}, {\cal B}({\mathbb T}))$ of tessellations is a Polish space (see \cite{Martinez-Nagel}), Kolmogorov's extension Theorem (see e.g. \cite{Klenke-Englisch}) yields that there is a process $(Y(t),t>0)$ with the respective finite-dimensional distributions.

\section{A class of cell division processes}
Inspired by Cowan's paper \cite{Cowan} we study a certain class of random tessellation processes which are generated by sequential cell division. But we will consider continuous-time processes only.
A cell division process is defined by the distributions of life times (that corresponds to Cowan's selection rule) and by a division rule for the extant cells.

For $h\in \mathcal{H}$  we denote by  $h^+$ and $h^-$  the closed half-spaces of $\mathbb{R}^d$ generated by $h$, with the following definition:
If $u \in {\cal S}^{d-1}_+$ is a vector in the upper half-sphere of ${\mathbb R}^d$ and $a \in {\mathbb R}$ such that $h=\{ x\in {\mathbb R}^d : \langle x,u \rangle = a \}$, then
$ h^+  :=  \{ x\in {\mathbb R}^d : \langle x,u \rangle \geq a \}$   and $ h^-  :=  \{ x\in {\mathbb R}^d : \langle x,u \rangle \leq a \}$. \\

{\em Assumptions:}

\begin{enumerate}
\item[(i)] Let $\lambda: \mathfrak{P} \rightarrow (0, \infty)$ be a function with
\begin{equation}\label{boundlambda}
\forall C\in \mathfrak{P},\,  \exists k_C \geq 1, \,  \forall \, h\in [C] :\,  \lambda (C\cap h^\pm )\leq k_C \lambda (C).
\end{equation}
\item[(ii)] Let  $\{\Lambda_{[C]}: C \in \mathfrak{P}\}$ be a family of probability measures on $({\cal H},{\mathfrak H})$ where $\Lambda_{[C]}$ is concentrated on $[C]$, but $\Lambda_{[C]}$ is not concentrated on a set of hyperplanes which are all parallel to one line (i.e. the directional distribution is not concentrated on a great subsphere, cf. \cite{SW}, Subsection 10.3)
\end{enumerate}

{\em The construction:}\\
Let $W \in \mathfrak{P}$, referred to as a window.
The cell division process $(Y(t,W): t \geq 0)$ with states in ${\mathbb T}\wedge W$  is defined by the following construction.

\begin{enumerate}

\item Let   $\tau_0, \tau_1, ... \sim \mathcal{E}(1)$ be a sequence of i.i.d. random variables which are exponentially distributed with parameter $1$.

\item  $Y(t,W) = \emptyset $ for all $0\leq t<\frac{1}{\lambda(W)}\tau_0$
 
\item  At time $t=\frac{1}{\lambda(W)}\tau_0$ the window $W$ is divided by a random hyperplane $h_0$ with law  $\Lambda_{[W]}$. Two new cells $C_1,\, C_2$ are born with $C_1=W\cap h_0^+$ and  $C_2=W\cap h_0^-$. The state of the process is $Y(t,W)= W\cap h_0$ for $\frac{1}{\lambda(W)}\tau_0 \leq t < \min \{\frac{1}{\lambda(C_1)}\tau_1,\, \frac{1}{\lambda(C_2)}\tau_2 \}$.
 
\item  Any cell $C_i$ which appears during the construction has a life time $\frac{1}{\lambda(C_i)} \tau_i$. At the end of its life time it is divided by a random hyperplane $h_i$ with the law $\Lambda_{[C_i]}$, into the new cells  $C_i \cap h_i^+$ and $C_i \cap h_i^-$, respectively. Always, $h_i$ is assumed to be conditionally independent of all the other dividing hyperplanes, given the  cell $C_i$. Thus the process $Y(t,W)$ jumps into another state exactly at those times when the life time of one of the extant cells elapses and the respective cell is divided.
 
\end{enumerate}

\begin{example}\label{exam}
Some examples for the functional $\lambda$ (corresponding  to Cowan's selection rule) are:

\begin{enumerate}
\item[S1]
$\lambda (C) =W_i (C)$, where $W_i$ denotes the $i$-th intrinsic volume (see e.g. \cite{SW}), $i=0,\ldots , d$. These functionals are monotonically increasing, i.e. 
$$\forall C,\, C'\in \mathfrak{P} : C'\subset C \Rightarrow \lambda (C')\leq  \lambda (C).$$
In particular $W_0(C)=1$ for all  $C\in \mathfrak{P}$ and $W_d$ is the volume.
 
\item[S2]
$\lambda (C) \ldots $ number of vertices of $C\in \mathfrak{P}$. This functional is not monotone in $C$, but condition (\ref{boundlambda}) is satisfied.

\item[S3]
For a given (non-zero) translation invariant and locally finite measure $\Lambda^*$ on $({\cal H},{\mathfrak H})$, chose $\lambda (C)=\Lambda^* ([C])$, $C\in \mathfrak{P}$.
If $\Lambda^*$ is also rotation invariant, then this functional coincides, up to a constant factor, with the intrinsic volume $W_1$.

\end{enumerate}

Some examples for the distributions $\Lambda_{[C]}$, $C\in \mathfrak{P}$ (correspond to Cowan's division rule)  are:

\begin{enumerate}
\item[D1] 
Let $\Lambda$ be a (non-zero) translation invariant and locally finite measure  on $({\cal H},{\mathfrak H})$ that is not concentrated on a set of hyperplanes which are all parallel to one line. For this measure define  $\Lambda_{[C]}(\cdot )=\Lambda (\cdot \cap [C])/ (\Lambda ([C]))$.

\item[D2] Define $\Lambda_{[C]}$  by the following procedure: throw a random point uniformly into $C$, and then choose a random hyperplane through this point with a certain directional distribution. 

\end{enumerate}

\end{example}

\begin{remark}
The homogeneous STIT tessellations as they were first introduced in \cite{Nagel:2005} fit into this scheme, choosing a translation invariant measure $\Lambda$ on $({\cal H},{\mathfrak H})$ and using  $\lambda (C)=\Lambda ([C])$, $\Lambda_{[C]}(\cdot )=\Lambda (\cdot \cap [C])/(\Lambda ([C]))$ for all $C\in \mathfrak{P}$.
\end{remark}

\section{Necessary and sufficient conditions for consistency}

\begin{theorem}\label{theornecess}
Let the family of tessellation processes $((Y(t,W), t > 0): W \in \mathfrak{P})$ be a family of cell division processes determined by  $\lambda$ and $\{\Lambda_{[C]}: C \in \mathfrak{P}\}$ which satisfy assumptions (i) and (ii) above. If this family of processes  is consistent in space then there exists a measure $\nu$  on $[\mathcal{H}, \mathfrak{H}]$ such that  for all $C\in \mathfrak{P}$
\begin{equation}\label{lambdanu}
\lambda (C)= \nu ([C]) 
\end{equation}
and
\begin{equation}\label{lambdagross}
\Lambda_{[C]} (\cdot ) = \frac{1}{\nu([C])} \nu (\cdot \cap [C]).
\end{equation}
\end{theorem}

Now we consider consistent families of cell division processes which yield homogeneous (i.e. spatially stationary) tessellations in ${\mathbb R}^d$. It was already known that homogeneous STIT tessellation processes are consistent. The following theorem states that STIT are the only consistent cell division processes.

\begin{theorem}\label{theorSTITonly}
Let the family of tessellation processes $((Y(t,W), t > 0): W \in \mathfrak{P})$ be a family of cell division processes determined by  $\lambda$ and $\{\Lambda_{[C]}: C \in \mathfrak{P}\}$ which satisfy assumptions (i) and (ii) above. This family of processes is consistent in space and all $Y_t$,  $t>0$, are homogeneous (spatially stationary), if and only if this process has the same distribution as the homogeneous STIT process driven by the hyperplane measure $\nu$ given in Theorem \ref{theornecess}.
\end{theorem}

\begin{remark}
It is easily seen by examples that the choices of $\lambda$ and $\{\Lambda_{[C]}: C \in \mathfrak{P}\}$ mentioned in the Example \ref{exam} and different from S3, D1 do not fulfill the necessary conditions for consistency.
\end{remark}

\section{Proofs}

\begin{lemma}\label{lemma_fundam}
Fix  $\lambda$ and $\{\Lambda_{[C]}: C \in \mathfrak{P}\}$ which satisfy (i) and (ii) above.  Let be $V, W \in \mathfrak{P}$ and $V \subset W$. If 
for all $t>0$ 
\begin{equation}\label{eq: vorcons}
Y(V,t)\stackrel{D}{=}Y(W,t)\cap V
\end{equation}
then for all $H \in \mathfrak{H}$, $H \subset [V]$

\begin{equation}\label{eq: Fundamentalgleichung-der-Konsistenz}
\lambda(V) \Lambda_{[V]}(H) = \lambda(W) \Lambda_{[W]}(H)
\end{equation}

\end{lemma}
\begin{proof}
Consider polytopes $V, W \in \mathfrak{P}$ with $V \subset W$ and  a Borel set $B \subset V$.\\
For the window $V$  the life time until the first division is $\frac{1}{\lambda(V)}\tau_0$. This division generates two new cells, $C_1, C_2$, say. Condition (\ref{boundlambda}) ensures that the waiting time until the next division in $V$ (i.e. a division of $C_1$ or $C_2$ respectively) is
$$
\min \{ \frac{1}{\lambda(C_1)}\tau_1, \frac{1}{\lambda(C_2)}\tau_2 \} \geq \min \{ \frac{1}{k_V\lambda(V)}\tau_1, \frac{1}{k_V\lambda(V)}\tau_2 \} 
= \frac{1}{k_V\lambda(V)} \min \{ \tau_1, \tau_2 \}, 
$$
i.e. it is greater or equal than an exponentially distributed random variable with parameter $2 k_V \lambda(V)$.

Hence, the time of the second division is greater or equal to the sum of two independent exponentially distributed random variables, the first one with parameter $\lambda(V)$ and the second one with parameter $2 k_V \lambda(V)$.

Then the properties of the exponential distribution yield for the construction within the window $V$ that for small $\Delta t >0$ up to a probability $o(\Delta t)$ not more than one division of $V$ takes place in the time interval $(0,\Delta t)$. Note the the event that $V$ is divided exactly once until $\Delta t$ can also be written as $Y(V,\Delta t)=h_0\cap V$. Hence

$$\begin{array}{rl}
  & \mathbb{P}(Y(V,\Delta t) \cap B \not= \emptyset)\\ &\\
= & \mathbb{P}(Y(V,\Delta t) \cap B \not= \emptyset ,\, Y(V,\Delta t)=h_0\cap V)\\ &\\
  &  + \mathbb{P}(Y(V,\Delta t) \cap B \not= \emptyset ,\, Y(V,\Delta t) \mbox{ contains more than two cells }) \\ &\\
= & \mathbb{P}(Y(V,\Delta t) \cap B \not= \emptyset ,\, Y(V,\Delta t)=h_0\cap V )  +   o(\Delta t) \\ &\\
= & \mathbb{P}(Y(V,\Delta t) \cap B \not= \emptyset  | Y(V,\Delta t)=h_0\cap V )\cdot P(Y(V,\Delta t)=h_0\cap V) +   o(\Delta t) \\ &\\
= & \Lambda_{[V]}([B]) \cdot (\lambda(V) \Delta t + o(\Delta t)) +  o(\Delta t)  \\ &\\
= & \Lambda_{[V]}([B]) \, \lambda(V)\, \Delta t + o(\Delta t)
\end{array}
$$

Analogously, for the same set $B$ we obtain for the construction in $W$
$$
   \mathbb{P}(Y(W,\Delta t) \cap B \not= \emptyset )
=  \Lambda_{[W]}([B])  \, \lambda(W) \, \Delta t + o(\Delta t).
$$ 

The identity (\ref{eq: vorcons}) implies that for $B\subset V$ and $\Delta t>0$ holds
$$
\mathbb{P}(Y(V,\Delta t) \cap B \not= \emptyset ) = \mathbb{P}(Y(W,\Delta t)\cap V \cap B \not= \emptyset ) = \mathbb{P}(Y(W,\Delta t) \cap B \not= \emptyset ) .
$$

Consequently, for the limits
$$\begin{array}{rcl}
    \Lambda_{[V]}([B] )\lambda(V)  
 &=&  \lim_{\Delta t \rightarrow 0}\frac{\mathbb{P}(Y(V,\Delta t) \cap B \not= \emptyset )}{\Delta t} 
 =  \lim_{\Delta t \rightarrow 0}\frac{\mathbb{P}(Y(W,\Delta t) \cap B \not= \emptyset )}{\Delta t}\\ &&\\ 
 &=&  \Lambda_{[W]}([B] )\lambda(W).
 \end{array}
$$
for all Borel sets $B \subset V \subset W$.\\

As the $\Lambda_{[V]}$ and $\Lambda_{[W]}$ are probability measures and $\Lambda_{[V]}([V])=1$,
$$\begin{array}{rl} & \lambda(V) \Lambda_{[V]}([V] \setminus [B])\\&\\
= & \lambda(V) (1 - \Lambda_{[V]}([B]))\\&\\
= & \lambda(V) - \lambda(V)\Lambda_{[V]}([B])\\&\\
= & \lambda(W) \Lambda_{[W]}([V]) - \lambda(W)\Lambda_{[W]}([B])\\&\\
= & \lambda(W) \Lambda_{[W]}([V] \setminus [B]).
\end{array}$$
Note that for $B_1,B_2\subset V$ holds
$([V] \setminus [B_1]) \cap ([V] \setminus [B_2]) = [V] \setminus ([B_1] \cup [B_2])$, and therefore
the set $\mathcal{D} = \{[V] \setminus [B]: B \subset V, B \textrm{ open}\}$ is a $\cap$-stable generating system for $\mathfrak{H} \cap [V]$, see  \cite[Lemma 2.1.1]{SW}
Thus, because both the measures $\lambda(V) \Lambda_{[V]}([V])$ and $\lambda(W) \Lambda_{[W]}([V])$ are finite and coincide on $\mathcal{D}$,  they are equal on  $\mathfrak{H} \cap [V]$.
\end{proof}

For a sequence

 $W_1 \subset W_2 \subset ...$ with all $W_n \in \mathfrak{P}$ and $\bigcup_{n=1}^\infty W_n = \mathbb{R}^d$
 we define a set function $\nu$ on $[\mathcal{H}, \mathfrak{H}]$ by 
\begin{equation}\label{eq: Definition-nu}
\nu(H) = \lim_{n \rightarrow \infty} \lambda(W_n) \Lambda_{[W_n]} (H \cap [W_n]), \quad \forall H \in \mathfrak{H}.
\end{equation}

\begin{lemma} The function
$\nu$ defined by (\ref{eq: Definition-nu}) does not depend on the particular choice of the sequence $(W_n, \, n\in {\mathbb N})$, and $\nu$ is a measure on $[\mathcal{H}, \mathfrak{H}]$.
\end{lemma}
\begin{proof}
Equation  (\ref{eq: Fundamentalgleichung-der-Konsistenz}) and the non-negativity of $\lambda(W_n)$ and $\Lambda_{[W_n]}$ yield
$$\begin{array}{rl} & \lambda(W_{n+1}) \Lambda_{[W_{n+1}]}(H \cap [W_{n+1}])\\&\\
= & \lambda(W_{n+1}) \Lambda_{[W_{n+1}]}(H \cap ([W_{n}] \cup ([W_{n+1}] \setminus [W_n])))\\&\\
= & \lambda(W_{n+1}) \Lambda_{[W_{n+1}]}(H \cap [W_{n}]) +  \lambda(W_{n+1}) \Lambda_{[W_{n+1}]}(H \cap ([W_{n+1}] \setminus [W_n]))\\&\\
= & \lambda(W_n) \Lambda_{[W_n]}(H \cap [W_n]) + \lambda(W_{n+1}) \Lambda_{[W_{n+1}]}(H \cap ([W_{n+1}] \setminus [W_n]))\\&\\
\geq & \lambda(W_n) \Lambda_{[W_n]}(H \cap [W_n]),\end{array}$$
i.e., the sequence $\lambda(W_n) \Lambda_{[W_n]} (H \cap [W_n])$ is monotonically increasing. 
Thus, the limit exists with the possibility of the limit being $\infty$.\\
In order to prove that this limit does not depend on the particular choice of the sequence of windows, consider two such monotone sequences $(W_n, \, n\in {\mathbb N})$ and $(W'_n, \, n\in {\mathbb N})$. Then for any $n\in {\mathbb N}$ there is a $m\in {\mathbb N}$ such that $W_n\subset W'_m$ and hence \\
$\lim_{n \rightarrow \infty} \lambda(W_n) \Lambda_{[W_n]} (H \cap [W_n]) \leq  \lim_{n \rightarrow \infty} \lambda(W'_n) \Lambda_{[W'_n]} (H \cap [W'_n])$. Exchanging the roles of $W_n$ and $W'_n$, it is seen that the opposite inequality holds as well, and hence both limits are equal.

Let $H_1, H_2, ... \in \mathfrak{H}$ be pairwise disjoint sets from $\mathfrak{H}$. Then
$$\begin{array}{rl} & \nu(\bigcup_{i=1}^\infty H_i)\\&\\
= & \lim_{n \rightarrow \infty} \lambda(W_n) \Lambda_{[W_n]}\left(\left(\bigcup_{i=1}^\infty H_i\right) \cap [W_n]\right)\\&\\
= & \lim_{n \rightarrow \infty} \lambda(W_n) \Lambda_{[W_n]}\left(\bigcup_{i=1}^\infty \left(H_i \cap [W_n]\right)\right)\\&\\
\stackrel{(a)}{=} & \lim_{n \rightarrow \infty} \lambda(W_n) \sum_{i=1}^\infty \Lambda_{[W_n]}(H_i \cap [W_n])\\&\\
= & \lim_{n \rightarrow \infty} \sum_{i=1}^\infty \lambda(W_n)  \Lambda_{[W_n]}(H_i \cap [W_n])\\&\\
\stackrel{(b)}{=} & \sum_{i=1}^\infty \lim_{n \rightarrow \infty} \lambda(W_n)  \Lambda_{[W_n]}(H_i \cap [W_n])\\&\\
= & \sum_{i=1}^\infty \nu(H_i).\end{array}$$
Here, (a) is correct because the $\Lambda_{[W_n]}$ are measures themselves. Equation (b) is due to the monotone convergence theorem.\\
Thus, the $\sigma$-additivity is proved and hence $\nu$ is a measure on $[\mathcal{H}, \mathfrak{H}]$.
\end{proof}

\begin{corollary}
For all $W \in \mathfrak{P}$ and $H \in \mathfrak{H}$ with $H \subset [W]$ $$\nu(H) = \lambda(W) \Lambda_{[W]}(H).$$
\end{corollary}
\begin{proof}
If $W \in \mathfrak{P}$ there is an $n_0$ such that $W \subset W_{n_0}$. Thus $H \subset [W] \subset [W_{n_0}]$ and hence for all $n \geq n_0$ equation (\ref{eq: Fundamentalgleichung-der-Konsistenz})
yields
$$\lambda(W) \Lambda_{[W]}(H) = \lambda(W_n) \Lambda_{[W_n]}(H)= \nu([W]).$$
\end{proof}

\begin{proof} {\em (of Theorem \ref{theornecess})}
Putting $H=[W]$, this corollary and (\ref{eq: Fundamentalgleichung-der-Konsistenz}) immediately yield equations (\ref{lambdanu}) and  (\ref{lambdagross}). 

\end{proof}

\begin{proof} {\em (of Theorem \ref{theorSTITonly})}
According to Theorem \ref{theornecess} the spatial consistency yields the existence of a measure $\nu$ which controls $\lambda$ and $\{\Lambda_{[C]}: C \in \mathfrak{P}\}$. Now it is sufficient to show that the homogeneity (in space) of the $Y(t)$ implies that  $\nu$ is translation invariant. This can be done analogously to the proof of Lemma \ref{lemma_fundam}, choosing a Borel set $B$, a translation vector $x \in \mathbb{R}^d$ and then a window $W$ such that $B,\ B+x \subset W$.
\end{proof}

\bibliographystyle{plain} \bibliography{literatur}

\end{document}